\documentclass[12pt,a4paper,reqno]{amsart}

\usepackage{graphicx, xcolor, enumerate}
\usepackage{amsmath, amssymb}
\usepackage{here}

\usepackage{amsthm} 
\usepackage{newtxmath,newtxtext}
\usepackage{subcaption,graphicx}
 \usepackage{bm}
\usepackage{lineno}
\usepackage{bigints}

\graphicspath{{./Figures/}}		

\theoremstyle{plain} 
\newtheorem{theorem}{Theorem}[section]
\newtheorem{lemma}[theorem]{Lemma}
\newtheorem{corollary}[theorem]{Corollary}

\newtheorem{fact}[theorem]{Fact}

\theoremstyle{definition}

\newtheorem{remark}[theorem]{Remark}
\newtheorem{example}[theorem]{Example}

 \usepackage{caption}

\usepackage{color}
\usepackage{graphicx,color}
\usepackage{amsmath, amssymb, graphics}

\newcommand{\rmG}{\mathrm{G}}

\newcommand{\rmx}{\mathrm{x}}
\newcommand{\rmy}{\mathrm{y}}

\begin{document}

\title[Bernstein-type theorem for CMC surfaces in $\mathbb{I}^3$]
{Bernstein-type theorem for constant mean curvature surfaces in the isotropic 3-space}

\author[S.Akamine]{Shintaro Akamine}
\address{Shintaro Akamine \\
College of Bioresource Sciences \\ Nihon University \\ 1866 Kameino, Fujisawa, Kanagawa, 252-0880, Japan}
\email{akamine.shintaro@nihon-u.ac.jp}

\author[W.Lee]{Wonjoo Lee}
\address{Wonjoo Lee \\ Department of Mathematics \\ Jeonbuk National University \\ 567 Baekje-daero, Deokjin-gu, Jeonju-si, Jeonbuk-do 54896, Republic of Korea}
\email{w$\_$lee@jbnu.ac.kr}

\author[S-D.Yang]{Seong-Deog Yang}
\address{Seong-Deog Yang \\ Department of Mathematics \\ Korea University \\ 145 Anam-ro, Seongbuk-gu, Seoul 02841, Republic of Korea}
\email{sdyang@korea.ac.kr}


\begin{abstract}
There are many non-trivial entire spacelike graphs with constant mean curvature $H$ (CMC $H$, for short) in the isotropic 3-space $\mathbb{I}^3$. In this paper, we show a value distribution theorem of Gaussian curvature of complete spacelike  constant mean curvature surfaces in $\mathbb{I}^3$, which implies a Bernstein-type theorem for CMC $H$ graphs in $\mathbb{I}^3$.
\end{abstract}

\maketitle


\section{Introduction}

Recently, the study of surfaces in degenerate spaces is gaining momentum. Many well known propositions of the classical surface theories are tested in degenerate spaces and new ones are being found. In this article, we focus on spacelike surfaces in the isotropic $3$-space $\mathbb{I}^3$, which can be identified with $\mathbb{R}^3(\ell, x, y)$ equipped with the degenerate metric $\operatorname{d}\!x^2+\operatorname{d}\!y^2$. 

Bernstein's problem, posed by Bernstein \cite{B, B2}, has attracted many mathematicians, and has been considered for various classes of hypersurfaces in various spaces, and there exist a plethora of excellent results related to it. Here we list only the minimum necessary references to give a rough historical background.

For the situation of vanishing mean curvature $H=0$, the original Bernstein theorem \cite{B, B2} states that any entire minimal graph in the Euclidean $3$-space $\mathbb{E}^3$ must be a plane. After that, in the 1960s, Bernstein's problem in $\mathbb{E}^n$ was also solved affirmatively when $n\leq 8$ by Fleming \cite{Fleming}, De Giorgi \cite{Giorgi}, Almgren \cite{Almgren} and Simons \cite{Simons}, and solved negatively when $n\geq 9$ by Bombieri, Giorgi and Giusti \cite{BGG}.
On the other hand, from the late 1960s, Calabi \cite{C} and Cheng-Yau \cite{CY} solved Bernstein's problem in the Lorentz-Minkowski $n$-space $\mathbb{L}^n$ affirmatively for arbitrary $n$, that is, any maximal hypersurfaces which is an entire graph must be a spacelike hyperplane. 

For the situation of constant mean curvature $H$ (CMC $H$, for short), Heinz \cite{H} proved in 1955 an estimation of $H$. As a corollary of Heinz's estimation, it was shown that any entire CMC $H$ graph must be a plane, that is, $H=0$. After that, in 1965-1966, Chern \cite{Chern} and Flanders \cite{F} generalized this result for hypersurfaces in $\mathbb{E}^n$. 
On the other hand, it was pointed out by Treibergs \cite{T} in 1982 that there are many entire non-zero CMC $H$ graphs in $\mathbb{L}^n$ other than the totally umbilic hyperboloid. 
Among such CMC $H$ entire graphs in $\mathbb{L}^n$, a Heinz-type estimation that characterize planes has also been studied recently by Honda-Kawakami-Koiso-Tori \cite{HKKT}.

In summary, Bernstein-type theorems can take various forms depending on whether the mean curvature $H$ is zero or not and what the ambient space is. 
Based on these studies, in this paper we give the following theorem which includes a Bernstein-type theorem for CMC $H$ surfaces in the isotropic $3$-space $\mathbb{I}^3$.

\begin{theorem}\label{thm:main}
If a connected complete spacelike constant mean curvature $H$ surface in $\mathbb{I}^3$ has non-constant Gaussian curvature $K$, then $K$ must take all values less than $H^2$ infinitely many times.  

In particular, if the Gaussian curvature $K$ of a connected complete spacelike constant mean curvature $H$ surface has an exceptional value less than $H^2$, then $K$ is constant, and the followings hold. 
\begin{itemize}
\item[$(1)$] When $H = 0$, the surface is either a plane or a rectangular hyperbolic paraboloid.
\item[$(2)$] When $H\ne0$, the surface is either a cylinder, an elliptic paraboloid, or a non-rectangular hyperbolic paraboloid.
\end{itemize}
\end{theorem}
 See Section~\ref{Sec.2} for the meaning of a rectangular hyperbolic paraboloid. 

The proof of Theorem \ref{thm:main} is given in Section \ref{sec3}. Since completeness of a surface in $\mathbb{I}^3$ induces the fact that the surface is an entire graph, we can show Bernstein-type theorems for constant mean curvature surfaces in $\mathbb{I}^3$ from Theorem \ref{thm:main}, see Corollaries \ref{cor1} and \ref{cor2}.
Finally, a PDE theoretical interpretation of Theorem \ref{thm:main} is also given in Section \ref{sec:pde}.


\section{Preliminaries}\label{Sec.2}

\subsection{Isotropic 3-space $\mathbb{I}^3$}

The isotropic 3-space $\mathbb{I}^3$ originates as one of the Calye-Klein geometries, and
can be identified with the  lightlike hyperplane
	\[
		\{  (\rmx_0, \rmx_1, \rmx_2, \rmx_3)  \in \mathbb{L}^4  : \rmx_0-\rmx_3=0 \} 
	\]
of the Lorentzian $4$-space
$\mathbb{L}^4 := \{ (\rmx_0, \rmx_1, \rmx_2, \rmx_3) : \rmx_0, \rmx_1, \rmx_2, \rmx_3 \in \mathbb{R}\} $ with the  metric
	\[
		\langle (\rmx_0, \rmx_1, \rmx_2, \rmx_3), (\rmy_0, \rmy_1, \rmy_2, \rmy_3) \rangle := -\rmx_0 \rmy_0 + \rmx_1 \rmy_1 + \rmx_2 \rmy_2 + \rmx_3 \rmy_3.
	\]

The map
$$
\mathbb{R}^3 \to \mathbb{L}^4, \qquad
(\ell,  x, y) \mapsto (\ell,  x, y,\ell)
$$
provides a global coordinate chart for the lightlike hyperplane,
through which we identify $\mathbb{I}^3$ with $\mathbb{R}^3$ equipped with the coordinates 
$\ell, x, y$, the pullback metric
	\begin{equation}\label{Eq:MetricOfI3}
		\operatorname{d}\!s^2=\operatorname{d}\!x^2+\operatorname{d}\!y^2,
	\end{equation}
and the maps 
	\begin{equation}\label{Eq:202505200927AM}
	\iota\begin{pmatrix} \ell \\ x \\ y \end{pmatrix} = \left( \begin{array}{ccc}
   s & k_1 & k_2 \\
   \multicolumn{1}{c}{%
     \begin{array}{@{} c @{}}
       0 \\
       0 
     \end{array}
   } &
   \multicolumn{2}{c}{%
     \begin{array}{@{} c @{}}
       \vcenter{\hbox{\scalebox{1}{$\mathcal{R}$}}}
     \end{array}
   }
	\end{array} \right) \begin{pmatrix} \ell \\ x \\ y \end{pmatrix} + \begin{pmatrix} \ell_0 \\ x_0 \\ y_0 \end{pmatrix}
	\end{equation}

for some $k_1,k_2,\ell_0,x_0,y_0\in\mathbb{R}$ and $\mathcal{R} \in O(2)$, $s\in\{-1,1\}$
as isometries. For details, see \cite{Sachs, SY, Silva1,Strubecker} for example.

\subsection{ Basic surface theory of $\mathbb{I}^3$}\label{Sec:2.2}

For any spacelike immersion $X\colon\Sigma \to \mathbb{I}^3$,
there is a unique map $\rmG \colon \Sigma \to \mathbb{L}^4$ with
	\[
		\langle \rmG, \rmG \rangle 
		= \langle \rmG, \operatorname{d}\!X \rangle 
		= \langle \rmG, \mathfrak{p} \rangle -1 = 0,
		\qquad\text{where}\quad \mathfrak{p} := (1,0,0,1).
	\]
We call $\rmG$ the {\it lightlike Gauss map} of $X$. We may assume without loss of generality that the metric can be written as
	\[
		\operatorname{d}\!s^2=e^{2{\sigma}}(\operatorname{d}\!u^2+\operatorname{d}\!v^2)=e^{2{\sigma}}\operatorname{d}\!z\operatorname{d}\!\bar{z},\qquad\text{where}\quad z := u+iv
	\]
for some coordinate system  $(u,v)$ and for some function $\sigma \colon \Sigma \to \mathbb{R}$. 
Then we have
	\[
		H =2e^{-2\sigma}\langle \rmG, X_{z\bar{z}} \rangle, 								\qquad
		Q\operatorname{d}\!z^2 = \langle \rmG, X_{zz} \rangle\operatorname{d}\!z^2,		\qquad
		K=H^2-4Q\bar{Q}e^{-4{\sigma}}
	\]
for the mean curvature, the Hopf differential, and the Gaussian curvature of $X$, respectively.
Note that $H^2 - K \ge 0$ and the Gauss equation is written as
	\[
		\sigma_{z\bar{z}}=0.
	\]
This implies that $\operatorname{d}\!s^2$ is a flat metric and hence the Gaussian curvature $K$ is not intrinsic.

For the graph of $\ell=f(x,y)$, $H$ and $K$ are written as
	\begin{equation}\label{eq:HK}
		H=\frac{1}{2}(f_{xx}+f_{yy}),\qquad K=f_{xx}f_{yy}-f_{xy}^2.
	\end{equation}
We remark that the graph of a smooth function is always spacelike. 

Note that the graph of $f(x,y) = \alpha x^2 + \beta y^2$ has $H=\alpha+\beta$ and $K=4\alpha \beta$. Its image by any of the isometries \eqref{Eq:202505200927AM} is called a {\it plane} if $\alpha=\beta=0$, a {\it cylinder} if $\alpha\beta=0$, $(\alpha,\beta)\neq (0,0)$, an {\it elliptic paraboloid} if $\alpha\beta>0$, a {\it hyperbolic paraboloid} if $\alpha\beta<0$, a {\it circular paraboloid} if $\alpha\beta>0$, $\alpha=\beta>0$,  a rectangular hyperbolic paraboloid if $\alpha\beta<0$, $\alpha=-\beta$.

\subsection{Weierstrass-type representation for constant mean curvature surfaces in $\mathbb{I}^3$}\label{sec:representation}

Let us recall the Weierstrass-type representation formula for constant mean curvature $H$ surfaces in $\mathbb{I}^3$ proved by Strubecker \cite{StrubeckerIII} for the case of $H=0$ and by Cho-Lee-Lee-Yang \cite{CLLY} for the case of $H\neq 0$.

\begin{fact}\label{fact:Wformula}
Any constant mean curvature $H$ immersion $X$ can locally be represented as 
	\begin{equation}\label{eq:Wformula}
	X(z)= \Re \int \left(\overline{h_1}+h_2, 1, -i\right)\omega, \quad h_1:=H\int \omega,
	\end{equation}
where $h_2$ is a holomorphic function, $\omega=\hat{\omega}\operatorname{d}\!z$ is a nowhere vanishing holomorphic 1-form. We call the pair $(h_2,\omega)$ the {\it Weierstrass data} of $X$.
\end{fact}

If we take $H=0$ in \eqref{eq:Wformula}, then we obtain a zero mean curvature surface. Hence, we can see that any constant mean curvature $H$ immersion $X=X_H$ is realized by the deformation $\{X_H\}_{H\in \mathbb{R}}$ of a zero mean curvature surface $X_0$ by changing the mean curvature $H$ in \eqref{eq:Wformula}. In this sense, we call $X_H$ the {\it CMC $H$ lift} of $X_0$.
 Furthermore, the first fundamental form of $X$ in \eqref{eq:Wformula} is $|\omega|^2=|\hat{\omega}|^2\operatorname{d}\!z\operatorname{d}\!\bar{z}$, hence the deformation $\{X_H\}_{H\in \mathbb{R}}$ is an isometric deformation.


\section{Proof of Theorem \ref{thm:main}}\label{sec3}

In this section, we give a proof of Theorem \ref{thm:main} and some important examples.
First, we prepare some Lemmas.

By using the representation formula \eqref{eq:Wformula} in Fact \ref{fact:Wformula}, we can directly check the following  relation between curvatures and Weierstrass data $(h_2,\omega)$.

\begin{lemma}\label{lemma}
The Gaussian curvature $K$ and the mean curvature $H$ of a constant mean curvature $H$ surface in $\mathbb{I}^3$ satisfy the following relation

	\begin{equation}\label{eq:KH}
	K=H^2-\left| \frac{\operatorname{d}\!h_2}{\omega} \right|^2.
	\end{equation}
\end{lemma}

\begin{lemma}\label{lemma2}
Any constant mean curvature $H$ surface with constant Gaussian curvature $K$ is a part of the surface
	\begin{equation}\label{eq:EHP}
	\ell = f(x,y)=H\left( \frac{x^2+y^2}{2}\right)+ \sqrt{H^2-K}\left( \frac{x^2-y^2}{2}\right)
	\end{equation}
up to an isometry of $\mathbb{I}^3$.
\end{lemma}

\begin{proof}
By Lemma \ref{lemma}, 
	\[
		\left|\frac{\operatorname{d}\!h_2}{\omega} \right|^2 = H^2-K.
	\]
Since the function $\operatorname{d}\!h_2/\omega $ is holomorphic, there exists a real number $\theta$ such that 
	\begin{equation}\label{eq:constancy}
	\frac{\operatorname{d}\!h_2}{\omega} =e^{i\theta}\sqrt{H^2-K}.
	\end{equation}

On the other hand, by the formula \eqref{eq:Wformula} we have the relation
	\[
		\int{\omega} = x+iy.
	\]
The relation $\hat{\omega}\neq 0$ implies that $x,y$ coordinates in $\mathbb{I}^3$ give local isothermal coordinates of the surface. Hence, we can take $z=x+iy$ as a complex coordinate so that $\omega=\operatorname{d}\!z$. 
For such a complex coordinate, $h_1$ and $h_2$ can be computed as
	\[
		h_1=Hz,\quad h_2=\sqrt{H^2-K}e^{i\theta}z
	\]
up to additive complex constants by \eqref{eq:Wformula} and \eqref{eq:constancy}.
Therefore, by using the representation formula \eqref{eq:Wformula} again, the coordinate $\ell$ of the surface can be computed as follows :
\begin{align*}
\ell &= \Re{\int{\overline{h_1}\omega}} + \Re{\int{h_2\omega}} \\ 
&=H\left(\frac{x^2+y^2}{2} \right) +\sqrt{H^2-K}\left[ \cos{\theta}\left(\frac{x^2-y^2}{2}\right)-\sin{\theta} xy\right].
\end{align*}
If we introduced new coordinates $(\tilde{x},\tilde{y})$ by the following rotation of the $xy$-plane :
	\[
		\begin{pmatrix} x \\ y \end{pmatrix} = \begin{pmatrix} \cos{\tilde{\theta}}  & -\sin{\tilde{\theta}} \\ \sin{\tilde{\theta}} & \cos{\tilde{\theta}} \end{pmatrix}
\begin{pmatrix} \tilde{x} \\ \tilde{y} \end{pmatrix}, \quad \tilde{\theta}:=-\frac{\theta}{2},
	\]
$\ell=\ell(x,y)$ can be written as
	\[
		\ell(\tilde{x},\tilde{y})=H\left(\frac{\tilde{x}^2+\tilde{y}^2}{2}\right)+\sqrt{H^2-K}\left(\frac{\tilde{x}^2-\tilde{y}^2}{2}\right)
	\]
which proves the assertion.
\end{proof}

If we let 
	\[
		\alpha := \frac{1}{2} \left( H + \sqrt{H^2-K}\right),	\qquad \beta := \frac{1}{2} \left( H - \sqrt{H^2-K}\right),
	\]
then $\ell(\tilde{x},\tilde{y}) = \alpha \tilde{x}^2 + \beta \tilde{y}^2$ and $\alpha\beta =K/4$, $\alpha + \beta = H$, hence the surface in \eqref{eq:EHP} is 

\begin{itemize}
\item[(a)]  a plane if  $K=0$ and $H=0$,
\item[(b)]  a cylinder if  $K=0$ and $H\neq 0$,
\item[(c)]  a hyperbolic  paraboloid for $K<0$, 
\item[(d)]  a rectangular  hyperbolic paraboloid if $K < 0$ and $H=0$,
 \item[(e)]  an elliptic  paraboloid for $K>0$,
\item[(f)] a (totally umbilic) circular  paraboloid if $K > 0$ and $H^2-K=0$.
\end{itemize}

The case (d) is a special case of (c) and the case (f) is a special case of (e).

\begin{proof}[{\bf Proof of Theorem \ref{thm:main}}]
Since the surface $S$ is complete, we can check that the projection $\pi\colon (\ell,x,y)\mapsto (x,y)$ gives an isometry from the surface $S$ onto its image $\pi(S)$ with the metric $\operatorname{d}\!x^2+\operatorname{d}\!y^2$ and hence $\mathbb{R}^2=\pi(S)$, see \cite[Theorem 5.1]{S1}. Therefore, the surface $S$ is an entire graph of the form $\ell=f(x,y)$ for some smooth function $f$ defined on $\mathbb{R}^2$.

 By Lemma \ref{lemma}, $H^2-K=|\varphi(z)|^2$ for some entire holomorphic function $\varphi$. If $K$ is non-constant, then so is $\varphi$. Therefore, by  Picard's little theorem, the image of $\varphi$ is either the whole complex plane or the whole complex plane minus a single point. This implies that $H^2-K$ takes all values greater than zero and hence the first assertion holds. 
 
 By the above argument, if the Gaussian curvature $K$ of the considered surface has an exceptional value less than $H^2$, then $K$ is constant. Therefore, by Lemma \ref{lemma2} and the above classification list (a) $\sim$ (f), the proof has been completed.
\end{proof}

As we saw in the above proof, completeness of a surface in $\mathbb{I}^3$ induces the fact that the surface is an entire graph. Hence, we obtain Bernstein-type theorems for constant mean curvature surfaces in $\mathbb{I}^3$ as follows.

\begin{corollary}\label{cor1}
If an entire zero mean curvature graph $\ell=f(x,y)$ has bounded Gaussian curvature $K$, then $K$ must be constant and the surface is either a plane or a rectangular hyperbolic paraboloid.
\end{corollary}

\begin{corollary}\label{cor2}
If an entire non-zero constant mean curvature $H$ graph $\ell=f(x,y)$ has bounded Gaussian curvature $K$, then $K$ must be constant, and the surface is a cylinder, the elliptic paraboloid, or a non-rectangular hyperbolic paraboloid.
\end{corollary}

\begin{remark}\label{rem:bound}
Regarding Corollary \ref{cor1} and Corollary \ref{cor2}, we give some remarks.
\begin{itemize}
\item If we assume the boundedness of $f$ instead of that of $K$, the conclusion is that $f$ must be constant,  i.e.,  the graph of $f$ is a horizontal plane.
\item Since any spacelike surface always satisfies the relation $K\leq H^2$, boundedness of $K$ means $C_0\leq K \leq H^2$ for some constant $C_0$.
\end{itemize}
\end{remark}

Just as there are many non-trivial entire spacelike constant mean curvature graphs in Lorentz-Minkowski space \cite{T}, there are infinitely many entire constant mean curvature graphs in $\mathbb{I}^3$. Corollary \ref{cor1} and Corollary \ref{cor2} imply that all surfaces except those listed in (a)$\sim$(f) have unbounded Gaussian curvature. Moreover, Theorem \ref{thm:main} asserts that the image of Gaussian curvature $K$ of a complete constant mean curvature $H$ surface is either 
	\[
		\{\text{single point}\},\quad (-\infty, H^2)\quad \text{or}\quad (-\infty, H^2].
	\]
These correspond to the cases of ``constant (or bounded) Gaussian curvature'', ``unbounded Gaussian curvature without umbilic points'', and ``unbounded Gaussian curvature with an umbilic point'', respectively. In the following, we will give specific examples belonging to these three cases.

\begin{example}[Enneper surface of order $n$ and its CMC $H$ lift]\label{ex1}
For a positive integer $n\geq 2$, let us consider Weierstrass data $(h_2,\omega)=(z^{n-1}, \operatorname{d}\!z)$ ($z\in \mathbb{C}$), which is the same as that of the Enneper surface of order $n$ in Euclidean space (see \cite[p.211]{DHS} for example). 

The CMC $H$ lift defined by Weierstrass data $(h_2,\omega)$ is the entire graph
	\[
		\ell =\frac{1}{n}\Re{z^n}+\frac{H}{2}(x^2+y^2), \quad z \in \mathbb{C}.
	\]
By using Lemma \ref{lemma}, the Gaussian curvature $K$ of the surface is 
	\[
		K=H^2-\left|(n-1)z^{n-2}\right|^2,\quad z\in \mathbb{C}.
	\]
 For $n=2$, the surface is nothing but the CMC $H$ lift of a rectangular hyperbolic paraboloid with constant curvature $K=H^2-1$ and without an umbilic point. See Figure \ref{Fig:Ennepern2}. For $n\geq 3$, the surface has unbounded $K$ and with an umbilic point at the origin $z=0$. See Figure \ref{Fig:Ennepern3}.

\begin{figure}[htb]
\vspace{-1.0cm}
\begin{center}
 \begin{tabular}{{c@{\hspace{-20mm}}c@{\hspace{-20mm}}c}}
\hspace{-15mm}   \resizebox{8.0cm}{!}{\includegraphics[clip,scale=0.30,bb=0 0 555 449]{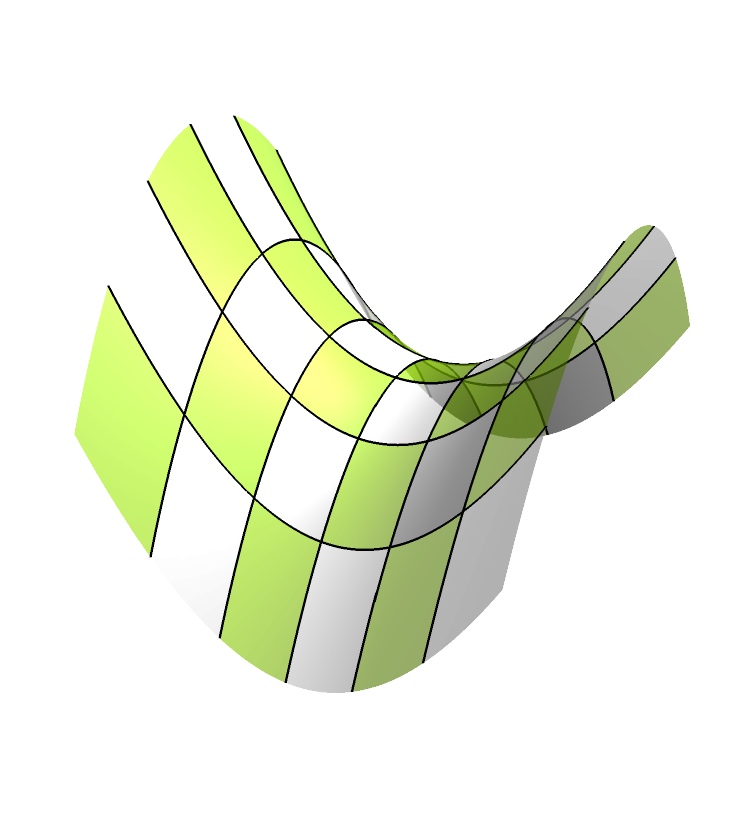}}&
 \hspace{-10mm}  \resizebox{8.0cm}{!}{\includegraphics[clip,scale=0.30,bb=0 0 555 449]{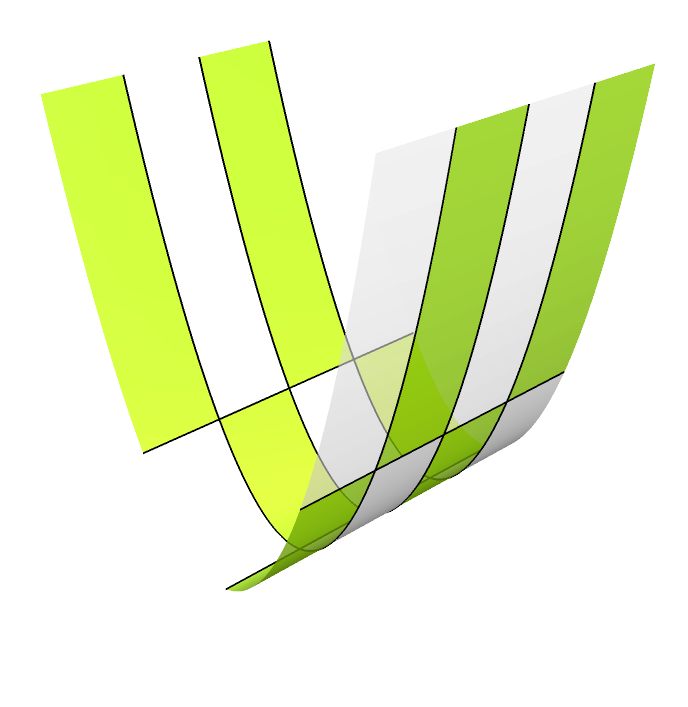}}&
 \hspace{-10mm} \resizebox{8.0cm}{!}{\includegraphics[clip,scale=0.30,bb=0 0 555 449]{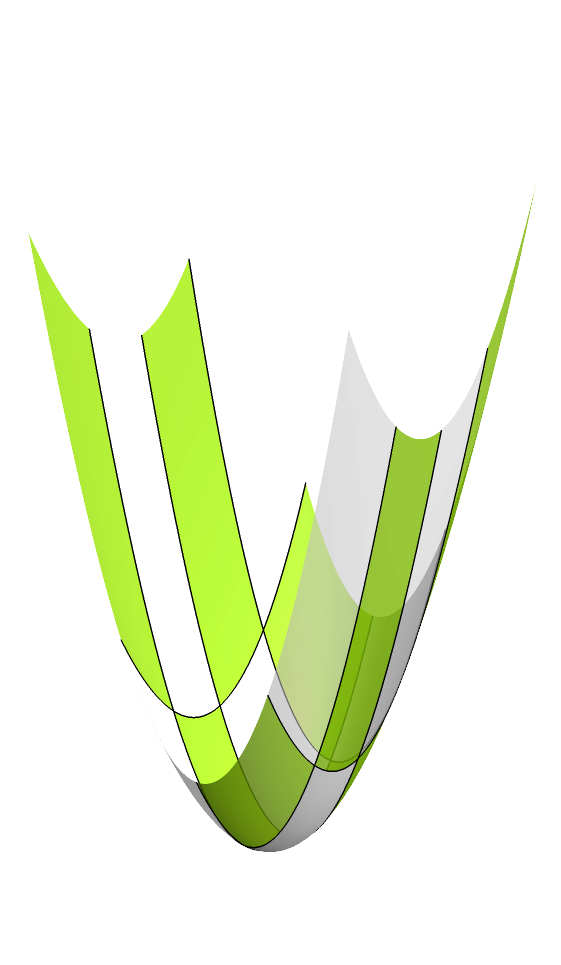}} \\
  {\hspace{-35mm}\footnotesize  Enneper surface with $H=0$} &
  {\hspace{-40mm}\footnotesize  CMC $H$ lift with $H=1$} &
  {\hspace{-45mm}\footnotesize  CMC $H$ lift with $H=2$}
 \end{tabular}
 \caption{Enneper surface of $n=2$ with $H=0$ (left) and its CMC lifts with $H=1$, which is a cylinder (center) and with $H=2$ (right). These surfaces are isometric entire constant mean curvature graphs with constant $K=H^2-1$. }
 \label{Fig:Ennepern2}
\end{center}
\end{figure}

\begin{figure}[htb]
\vspace{-1.0cm}
\begin{center}
 \begin{tabular}{{c@{\hspace{-20mm}}c@{\hspace{-20mm}}c}}
\hspace{-15mm}   \resizebox{8.0cm}{!}{\includegraphics[clip,scale=0.30,bb=0 0 555 449]{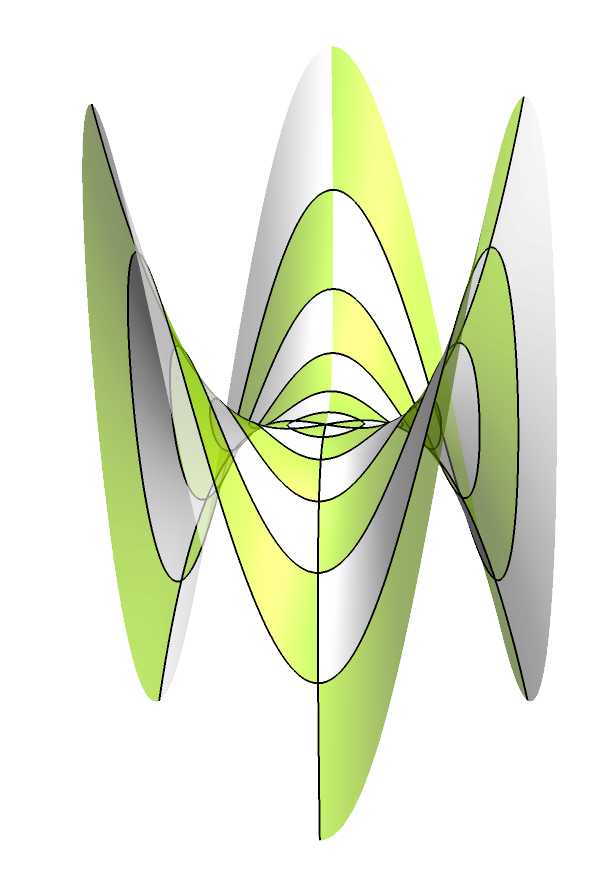}}&
 \hspace{-10mm}  \resizebox{8.0cm}{!}{\includegraphics[clip,scale=0.30,bb=0 0 555 449]{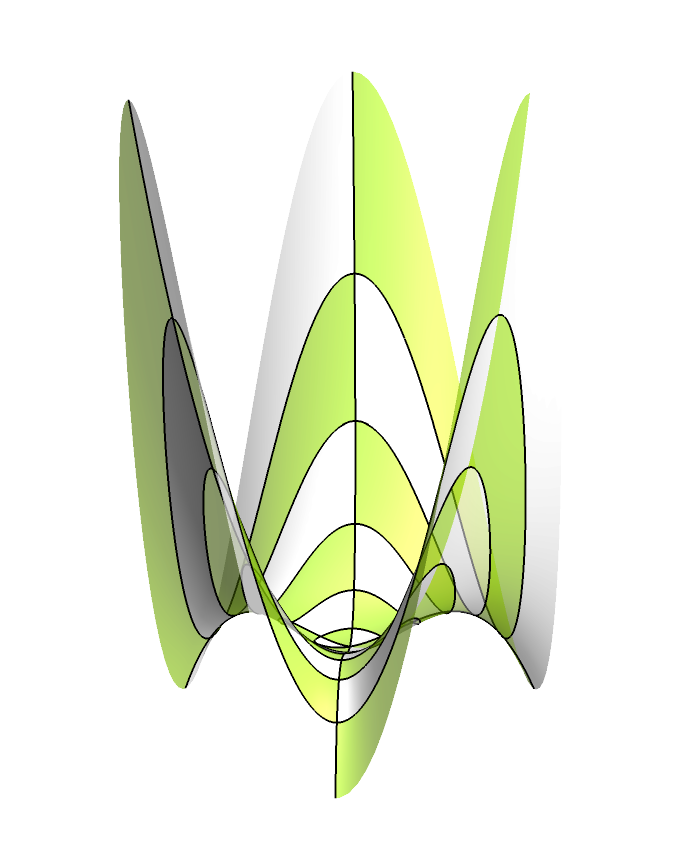}}&
 \hspace{-2mm} \resizebox{8.0cm}{!}{\includegraphics[clip,scale=0.30,bb=0 0 555 449]{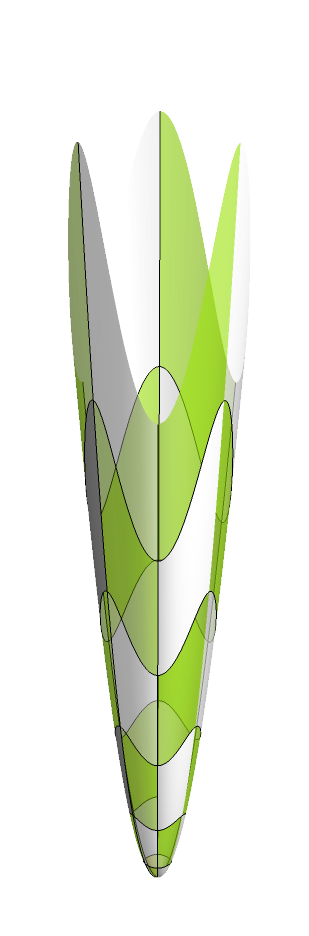}} \\
  {\hspace{-35mm}\footnotesize  Enneper surface with $H=0$} &
  {\hspace{-35mm}\footnotesize  CMC $H$ lift with $H=1.5$} &
  {\hspace{-55mm}\footnotesize  CMC $H$ lift with $H=10$}
 \end{tabular}
 \caption{Enneper surface of order $n=3$ with $H=0$ (left) and its CMC lifts with $H=1.5$ (center) and with $H=10$ (right). These surfaces are isometric entire constant mean curvature graphs with unbounded $K$.}
 \label{Fig:Ennepern3}
\end{center}
\end{figure}
\end{example}

\begin{example}[Entire constant mean curvature $H$ graphs without umbilic points]\label{ex:noumbilic}
Let us take Weierstrass data $(h_2,\omega)=(e^z, \operatorname{d}\!z)$ ($z\in \mathbb{C}$) in \eqref{eq:Wformula}, then we obtain a singly periodic zero mean curvature entire graph and its CMC $H$ lift :
	\[
		X_H(z)=\left( e^x\cos{y} +\frac{H}{2}(x^2 + y^2),x,y\right),\quad z=x+iy \in \mathbb{C}.
	\]
By Lemma \ref{lemma},  $K=H^2-|e^z|^2$ holds and hence $X_H$ does not have an umbilic point. See Figure \ref{Fig:Noumbilic}. 

\begin{figure}[htb]
\begin{center}
 \begin{tabular}{{c@{\hspace{-20mm}}c@{\hspace{-20mm}}c}}
\hspace{-5mm}   \resizebox{8.0cm}{!}{\includegraphics[clip,scale=0.30,bb=0 0 555 449]{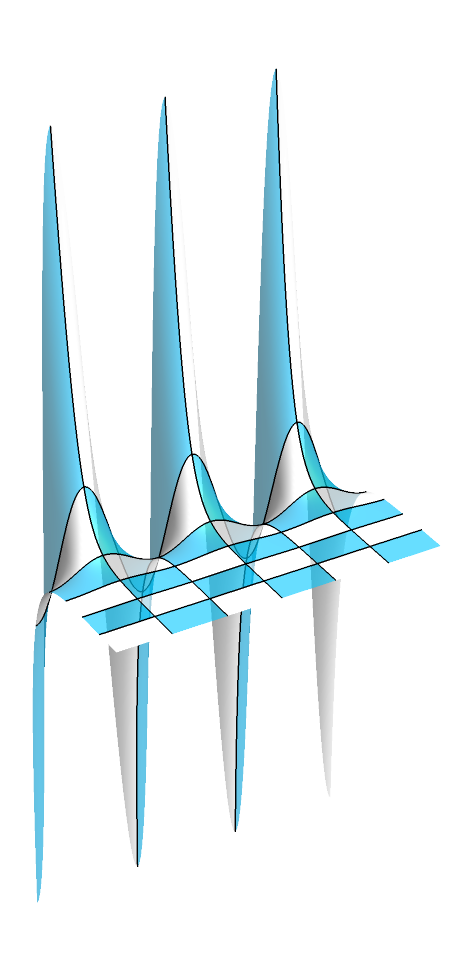}}&
 \hspace{-10mm}  \resizebox{8.0cm}{!}{\includegraphics[clip,scale=0.30,bb=0 0 555 449]{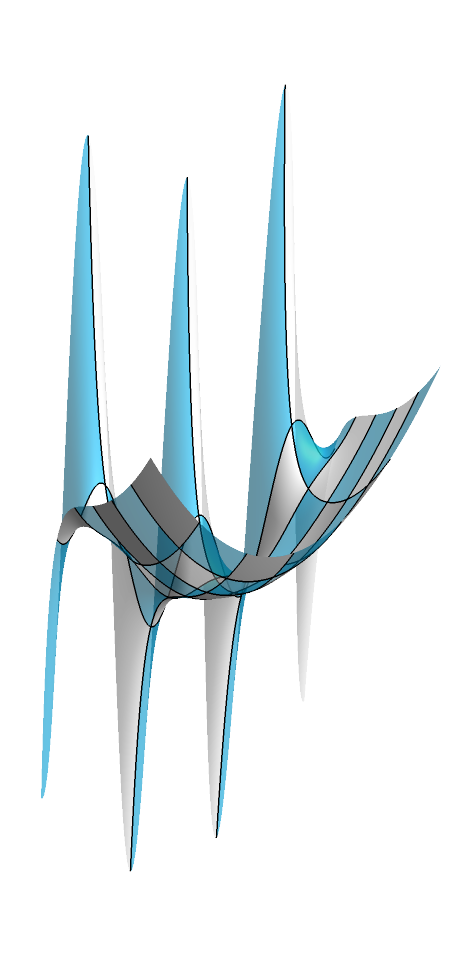}}&
 \hspace{-10mm} \resizebox{8.0cm}{!}{\includegraphics[clip,scale=0.30,bb=0 0 555 449]{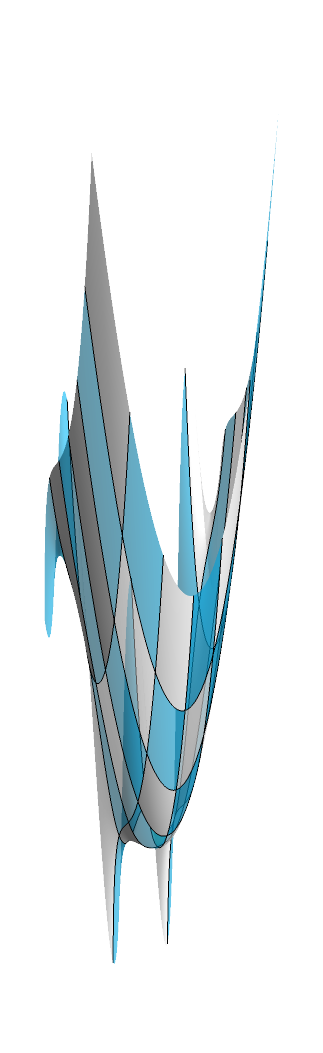}} \\
  {\hspace{-55mm}\footnotesize  Entire graph with $H=0$} &
  {\hspace{-60mm}\footnotesize   CMC $H$ lift with $H=0.1$} &
  {\hspace{-65mm}\footnotesize  CMC $H$ lift with $H=0.5$}
 \end{tabular}
 \caption{Entire constant mean curvature $H$ graphs $X_H$ without umbilic points.}
 \label{Fig:Noumbilic}
\end{center}
\end{figure}
\end{example}

At the end of this section, we present some corresponding results on the behavior of the Gaussian curvature of constant mean curvature surfaces in $\mathbb{L}^3$ in order to demonstrate that Theorem \ref{thm:main} is interesting phenomenon specific to isotropic space.

\begin{remark}
Regarding a value distribution of $K$ for constant mean curvature hypersurfaces in $\mathbb{L}^n$, Cheng-Yau \cite{CY} and Treibergs \cite{T} (see also Choi-Treibergs \cite{CT} and Wan \cite{Wan}) showed that any complete spacelike constant mean curvature  hypersurface in $\mathbb{L}^n$ has non-positive sectional curvature.

Especially in the $\mathbb{L}^3$ case, this means that the Gaussian curvature $K$ is non-positive. In addition, for a spacelike constant mean curvature $H$ surface in $\mathbb{L}^3$, the following relation holds.
	\[
		H^2+K = \left( \frac{\kappa_1+\kappa_2}{2}\right)^2+\left(-\kappa_1\kappa_2 \right) =  \left( \frac{\kappa_1-\kappa_2}{2}\right)^2\geq 0,
	\]
where $\kappa_1$ and $\kappa_2$ are the principal curvatures of the surface.
Combining these two relations, we obtain the estimation
	\[
		-H^2\leq K \leq 0,
	\]
that is, any complete spacelike constant mean curvature $H$ surface in $\mathbb{L}^3$ has bounded Gaussian curvature. 
\end{remark}

Moreover, Milnor \cite{Milnor} and Yamada \cite{Yamada} (see also Kawakami-Satake \cite{KS}) showed that non-existence of umbilic points strongly affects of global shapes of such surfaces as follows.

\begin{fact}\label{fact:inL3}
If a complete spacelike surface in $\mathbb{L}^3$ with constant mean curvature $H\neq 0$ satisfies $H^2 + K\geq  \varepsilon$ for some $\varepsilon>0$, then the surface is the hyperbolic cylinder.
\end{fact}

Here, the hyperbolic cylinder is the flat surface $-x_0^2 +x_1^2 = -\frac{1}{4H^2}$ for $x_0>0$ in $\mathbb{L}^3$ with the coordinates $(x_0,x_1,x_2)$. By Fact \ref{fact:inL3}, we can see that for any complete spacelike CMC $H$ surface in $\mathbb{L}^3$ which is not totally umbilic and not the hyperbolic cylinder, $K$ takes values arbitrary close to $-H^2$. 
This is remarkably different from the behavior of Gaussian curvature of surfaces in isotropic space discussed in this section.


\section{Relating results from a viewpoint of PDE}\label{sec:pde}
As we saw in \eqref{eq:HK}, the curvatures $H$ and $K$ of a graph $\ell=f(x,y)$ in $\mathbb{I}^3$ are written as the Laplacian $\Delta f$ and the hessian determinant $\mathcal{H}_f$ of $f$, respectively :
	\[
		2H=\Delta f\ (:=f_{xx}+f_{yy}),\quad K=\mathcal{H}_f\ (:=f_{xx}f_{yy}-f_{xy}^2).
	\]
We examine the results in this article from a viewpoint of PDE.  
First, the following result proved by Reilly \cite{Reilly} is also obtained as a corollary of Theorem \ref{thm:main}. 

\begin{fact}[Reilly]
Any entire solution $f=f(x,y)$ of $\Delta f =2H$ for a constant $H$ with bounded hessian determinant $\mathcal{H}_f$ is a quadratic polynomial.
\end{fact}

Combining this with Theorem \ref{thm:main}, we obtain the following :

\begin{corollary}
Any entire solution $f=f(x,y)$ of $\Delta f =2H$ for a constant $H$, the image of the hessian determinant $\mathcal{H}_f$ of $f$ is either $\{\text{single point}\}$, $(-\infty, H^2)$ or $(-\infty, H^2]$. Moreover, if the image of the hessian determinant of $f$ is a single point, then $f$ is written as \eqref{eq:EHP} up to a sign, a coordinate change and a linear function of $x$ and $y$.
\end{corollary}

 So far, we have focused on surfaces in $\mathbb{I}^3$ with constant mean curvature. Now we turn our attention to surfaces in $\mathbb{I}^3$ with constant Gaussian curvature. For example, the following J\"orgens' theorem \cite{Jorgens} on solutions of the Monge-Amp\`ere equation is well-known.

\begin{fact}[J\"orgens]
Any entire solution $f=f(x,y)$ of $f_{xx}f_{yy}-f_{xy}^2=K$ for a positive constant $K$  is a quadratic polynomial.
\end{fact}

Combining this with \cite[Theorem 5.1]{S1} as in the proof of Theorem \ref{thm:main} yields the following :

\begin{corollary}
If a connected complete spacelike surface in $\mathbb{I}^3$ has positive constant Gaussian curvature $K$, then the surface is the entire graph of a quadratic polynomial. In particular, the surface is an elliptic paraboloid.
\end{corollary}


\section*{Acknowledgement}
%
The first author was supported by JSPS KAKENHI Grant Numbers 23K12979 and 24K06709, and by the Research Institute for Mathematical Sciences (RIMS) at Kyoto University. The second and third authors were supported by the NRF of Korea funded by MSIT (Korea-Austria Scientific and Technological Cooperation RS-2025-1435299, P.I.: Joseph Cho), and the third author was also supported by MSIT (Basic Research Project, Young Scientist Grants (Type A) RS-2026-25500032, P.I.: Joseph Cho). The authors would like to thank Professor Yu Kawakami for informing them of the result in Remark~\ref{fact:inL3} and bringing the related references to their attention.
 

\end{document}